\documentclass{amsart}
\usepackage{bbm}
\usepackage{xypic}
\usepackage{amssymb}
\usepackage{dsfont}
\usepackage{hyperref}
\usepackage{booktabs}
\usepackage{thmtools}
\usepackage{thm-restate}
\usepackage{cleveref}
\usepackage[table]{xcolor}
\usepackage{enumitem}
\usepackage{caption}
\usepackage[utf8]{inputenc}
\usepackage[T1]{fontenc}

\usepackage{comment}
\usepackage{amsmath}
\usepackage{enumitem}

\usepackage{tikz-cd}

\xyoption{all}

\newif\iffurther
\furthertrue

\numberwithin{equation}{section}
\numberwithin{figure}{section}
\theoremstyle{plain}
\newtheorem{thm}{Theorem}[section] 
\newtheorem*{thm*}{Theorem}

\newtheorem{prop}[thm]{Proposition}
\newtheorem{cor}[thm]{Corollary}
\theoremstyle{definition}

\newtheorem{defn}[thm]{Definition}

\newtheorem{lem}[thm]{Lemma}

\theoremstyle{remark}
\newtheorem{rem}[thm]{Remark}

\newtheorem{exmpl}[thm]{Example}

\newtheorem*{acknowledgement*}{Acknowledgement}

\newcommand\suchthat{\;\ifnum\currentgrouptype=16 \middle\fi|\;}

\def\Im{{\operatorname{Im}}}

\newcommand{\Z}{\mathbb{Z}}

\usepackage{mathtools}
\DeclarePairedDelimiter\floor{\lfloor}{\rfloor}
\DeclarePairedDelimiter\ceil{\lceil}{\rceil}

\begin{document}

\title{How Many Reflections Make a Dihedral Set Large?}

\author{Be'eri Greenfeld} \address{Department of Mathematics and Statistics, Hunter College CUNY, 695 Park Avenue, New York NY 10065, USA} \email{beeri.greenfeld@hunter.cuny.edu}

\author{George King} 
\address{University of Washington, Department of Mathematics, Box 354350, C-138 Padelford Hall
Seattle, WA 98195-4350 USA}
\email{gking3@uw.edu}

\author{Xiaoxuan Li}
\address{University of Illinois Urbana-Champaign, Siebel School of Computing and Data Science, The Grainger College of Engineering, Thomas M. Siebel Center for Computer Science 201 North Goodwin Avenue, Urbana, IL 61801-2302A}
\email{axxwtgst@gmail.com}

\author{Sam Tacheny} 
\address{University of Washington, Department of Mathematics, Box 354350, C-138 Padelford Hall
Seattle, WA 98195-4350 USA}
\email{samueltacheny@gmail.com}

\thanks{This project has originated from the Washington eXperimental Math Lab (WXML) at the University of Washington under the mentorship of the first-named author.}

\begin{abstract}
Given a size-$k$ subset $S$ of a group $G$, how large can the product set $S^n$ be? 
We study this question, at several layers of refinement, for the infinite dihedral group.

First, we give an explicit formula for the maximum size of $S^n$ among all size-$k$ subsets with a prescribed number of reflections. 
We then determine the optimal number of reflections that a size-$k$ set should contain in order to maximize $|S^n|$. 

When $k$ is fixed and $n\to\infty$, we obtain a clean asymptotic expression for the maximal size of $S^n$.
Moreover, we compute this asymptotic separately for each fixed number of reflections in $S$. 
We show that the number of reflections influences the asymptotic size of $S^n$ only through a multiplicative coefficient, which admits a direct probabilistic interpretation.

Finally, we compute the growth exponent of the maximum of $|S^n|$ when~$k=~n$.
\end{abstract}

\maketitle

\section{Introduction}

Let $G$ be a group and let $S\subseteq G$ be a subset. Denote $S^n=\{s_1\cdots s_n|s_1,\dots,s_n\in~S\}$. If $S$ is a finite, symmetric, generating subset of $G$ containing $1$ then the function $\gamma_{G,S}(n)=|S^n|$ is the growth function of $G$; this is the most important large-scale geometric invariant associated with a finitely generated group. Gromov's celebrated theorem \cite{Gro} asserts that a finitely generated group has a polynomial growth if and only if it is virtually nilpotent. 
The space of growth functions of groups is still widely mysterious; see e.g. \cite{BartholdiErschler14,BartholdiErschler12,ErschlerZheng20,Gri84,Grigorchuk90,KassabovPak13,ShalomTao10} and references therein.

Now fix a number $k$ and consider all possible subsets of $G$ of cardinality $k$:
\[
\gamma_G(k,n) = \max_{\substack{S\subseteq G, \ |S|=k}} |S^n|.
\]
This can be thought of as a `uniform' version of the notion of growth, applicable to an arbitrary (not necessarily finitely generated) group; see \cite{Bozejko,delaHarpe,GL,Mann,SSh,Sh}. 
Notice that $\gamma_G(k,n)\leq k^n$. Generically, $\gamma_G(k,n)=k^n$ for all $k,n$: by a theorem of Semple and Shalev \cite{SSh} (relying on Zelmanov's resolution of the Restricted Burnside Problem), this is the case if $G$ is a finitely generated residually finite group, unless $G$ is virtually nilpotent. 
If $G$ is nilpotent then $\gamma_G(k,n)$ grows polynomially in $n$ (for any fixed $k$), and conversely, if $\gamma_G(k,n)$ is polynomial in $n$ then it is locally virtually nilpotent in a strong uniform sense \cite{Mann}.
If we let $k\rightarrow \infty$ as a function of $n$, the asymptotics are more mysterious. In \cite{GL} it is proved that, if $G$ is residually finite and not virtually abelian, then $\gamma_G(n,n)\geq (e^{-1/4} - o(1))n^n$ and this lower bound is optimal for the Heisenberg group. If $G$ is virtually abelian then $\gamma_G(n,n)$ grows exponentially in $n$.

Given a group $G$, there are two types of questions of interest: how fast can subsets of $G$ grow (namely, what is $\gamma_G(k,n)$); and what is the shape of a rapidly growing subset (namely, which size-$k$ subsets $S \subseteq G$ maximize $|S^n|$). 
For $G=\mathbb{Z}$ (or any abelian group of infinite exponent) then $\gamma_{G}(k,n)=\binom{n+k-1}{k-1}$. In particular, $\gamma_G(k,n)\asymp n^{k-1}$ as $k$ is fixed, and $\gamma_G(n,n)=(4\pm o(1))^{n}$. 
Indeed, given any $k,n$, a generic size-$k$ subset $S \subset \mathbb{Z}$ satisfies $|S^n|=\binom{n+k-1}{k-1}$.

In this paper, we focus on the infinite dihedral group $$D_\infty=\mathbb{Z}\rtimes_\tau \mathbb{Z}/2,\ \ \ \tau(m)=-m$$ the isometry group of $\mathbb{Z}$. Although this group is, in some sense, the closest possible non-abelian group to $\mathbb{Z}$, the computation of $\gamma_{D_\infty}(k,n)$ turns out to be very much different and much harder. 

Elements in $D_\infty$ can be classified as either translations ($x\mapsto x+d$) or reflections ($x\mapsto -x+d$). This time, there are $k+1$ types of generic size-$k$ subsets, depending on the number of reflections vs. translations among the $k$ chosen elements. We let
$$\gamma_{D_\infty}(k,p,n) = \max \Big\{|S^n|\ :\ S\subset D_\infty,\ S\ \text{has}\ p\ \text{reflections} \ \text{and}\ k-p\ \text{translations}\Big\}.$$
Evidently, $\gamma_{D_\infty}=\max_{0\leq p\leq k} \gamma_{D_\infty}(k,p,n)$. For $p=0$, we have $\gamma_{D_\infty}(k,0,n)=\gamma_{\mathbb{Z}}(k,n)=\binom{n+k-1}{k-1}$.
Our first main result is an explicit computation of this function. 
\begin{thm} \label{thm:k p n}
Let $k,p,n \in \mathbb{N}$ with $1 \leq p \leq k$. Then    
\begin{align*}
        \gamma_{D_\infty}(k,p,n) 
        = \binom{n + k - p - 1}{k - p - 1} & +  \sum_{\substack{0 \leq l < n \\ n \equiv l \bmod{2}}} N(k-p,l) + \sum_{\substack{1 \leq l \leq n \\ n \equiv l \bmod{2}}}\sum_{q=1}^{l}R(p,q)N(k-p,l-q).
\end{align*}
\end{thm}
\noindent 
Here $N(s,t)$ counts $s$-dimensional integral vectors of $l^1$-norm $t$, and $R(s,t)$ counts such vectors whose positive and negative entries are balanced; see precise formulas in the following sections.

Theorem \ref{thm:k p n} allows us to compute, for each $k,n$, the optimal number of reflections that a size-$k$ subset $S\subset D_\infty$ should contain to maximize $|S^n|$, leading to the curious pattern shown in the table below. The table shows the optimal number of reflections for $D_\infty$ (the shaded cells represent the diagonal $k=n$).

\footnotesize
\begin{center}
\begin{tabular}{c|ccccccccc} 
\toprule
$n \backslash k$ & 2 & 3 & 4 & 5 & 6 & 7 & 8 & 9 & 10 \\
\midrule
2 & \cellcolor{lightgray}1 & 1, 2 & 2, 3 & 3, 4 & 4, 5 & 5, 6 & 6, 7 & 7, 8 & 8, 9 \\
3 & 1 & \cellcolor{lightgray}1 & 2 & 2, 3 & 3 & 4 & 5 & 5 & 6 \\
4 & 1 & 1 & \cellcolor{lightgray}1 & 2 & 3 & 3 & 4 & 4 & 5 \\
5 & 1 & 1 & 1 & \cellcolor{lightgray}1 & 2 & 3 & 3 & 4 & 4 \\
6 & 1 & 1 & 1 & 1 & \cellcolor{lightgray}2 & 2 & 3 & 3 & 3 \\
7 & 1 & 1 & 1 & 1 & 1 & \cellcolor{lightgray}2 & 2 & 3 & 3 \\
8 & 1 & 1 & 1 & 1 & 1 & 2 & \cellcolor{lightgray}2 & 2 & 3 \\
9 & 1 & 1 & 1 & 1 & 1 & 1 & 2 & \cellcolor{lightgray}2 & 2 \\
10 & 1 & 1 & 1 & 1 & 1 & 1 & 1 & 2 & \cellcolor{lightgray}2 \\
11 & 1 & 1 & 1 & 1 & 1 & 1 & 1 & 2 & 2 \\
12 & 1 & 1 & 1 & 1 & 1 & 1 & 1 & 1 & 2 \\
13 & 1 & 1 & 1 & 1 & 1 & 1 & 1 & 1 & 1 \\
14 & 1 & 1 & 1 & 1 & 1 & 1 & 1 & 1 & 1 \\
15 & 1 & 1 & 1 & 1 & 1 & 1 & 1 & 1 & 1 \\
16 & 1 & 1 & 1 & 1 & 1 & 1 & 1 & 1 & 1 \\
17 & 1 & 1 & 1 & 1 & 1 & 1 & 1 & 1 & 1 \\
18 & 1 & 1 & 1 & 1 & 1 & 1 & 1 & 1 & 1 \\
19 & 1 & 1 & 1 & 1 & 1 & 1 & 1 & 1 & 1 \\
20 & 1 & 1 & 1 & 1 & 1 & 1 & 1 & 1 & 1 \\
\bottomrule
\end{tabular}
\end{center}

\normalsize

As hinted in the above table, for each fixed $k$ and $n\gg_k 1$, the optimal number of reflections is one.

\begin{thm} \label{thm:k fixed}
Let $k,p,n \in \mathbb{N}$ with $1 \leq p \leq k$ fixed. Then,
\begin{align*}
\gamma_{D_\infty}(k,p,n) & = \frac{\binom{2p-2}{p-1}}{2^{2p-2}} \frac{2^{k-1}}{(k-1)!} n^{k-1} + O(n^{k-2}) \\ &= \Pr\left(\substack{\text{A random subset of}\ \\ {\{1,\dots,2p-2\}\ \text{has size}\ p-1}}\right) \cdot \frac{2^{k-1}}{(k-1)!} n^{k-1} + O(n^{k-2})
\end{align*}
and, in particular,
$$\gamma_{D_\infty}(k,n)=\gamma_{D_\infty}(k,1,n)=\frac{2^{k-1}}{(k-1)!} n^{k-1} + O(n^{k-2}).$$
\end{thm}

Finally, we derive a computation of the growth exponent for $k=n$, which differs significantly from the case of $\mathbb{Z}$: 
\begin{thm} \label{thm:k=n}
We have
$\overline{\lim}_{n\rightarrow \infty} \gamma_{D_\infty}(n,n)^{1/n} = 3+2\sqrt{2}$.
\end{thm}

The difference from the case of $\mathbb{Z}$, and hence the difficulty, can be understood conceptually as follows. Let $Var(G)$ be the variety of $G$, that is, the class of all groups satisfying the same group laws as $G$ and let $F_k(Var(G))$ be the $k$-generated relatively free group in that variety with free generators $X=\{x_1,\dots,x_k\}$. For any subset $S\subseteq G$ of size $k$ there is a natural map $F_k(Var(G))\rightarrow G$ carrying $X$ to $S$, and thus $\gamma_G(k,n)\leq |X^n|$. Equality holds in some cases, including when $F_k(Var(G))$ is residually-$G$, or if the Cayley graph of $F_k(Var(G))$ is a limit of Cayley graphs of $G$ (see \cite{BE,bnnn,GL,Neumann67} for more details). This provides access to compute $\gamma_G(k,n)$ for abelian groups with infinite exponent (where $F_k(Var(G))\cong \mathbb{Z}^k$ and $X=\{\vec{e}_1,\dots,\vec{e}_k\}$, from which one easily concludes that $\gamma_G(k,n)=\binom{n+k-1}{k-1}$) or for some nilpotent groups (as utilized in \cite{GL}). However, equality might fail, even for $G=D_\infty$. Specifically, $\gamma_{D_\infty}(2,3)=6$ while $|X^3|=8$ in $F_2(Var(D_\infty))$.

\section{The infinite dihedral group} 

Recall that the infinite dihedral group is $D_\infty = \Z \rtimes \Z_2$, where $\Z_2 = \{\pm 1\} \curvearrowright \mathbb{Z}$ by multiplication, and the group operation can be written as
    \[
    (a_1, b_1)(a_2, b_2) = (a_1 + b_1a_2, b_1b_2).
    \]
The sign $b\in \{\pm 1\}$ of an element $(a,b)\in D_\infty$ can be thought of as its orientation, acting on $\mathbb{R}$.
We call an element with positive orientation a \emph{translation}, and an element with negative orientation a \emph{reflection}.

We start with the following combinatorial notions. We fix a $k$-tuple $S=(s^{(1)},\dots,s^{(k)}) \in D_{\infty}^k$ and write each $s^{(i)} = (a_i,b_i)$ with $a_i\in \mathbb{Z},\ b_i\in \{\pm 1\}$. The following definition requires to specify only a vector of signs $(b_1,\dots,b_k)\in \{\pm 1\}^k$, but in practice, it will be taken to be the sign vector corresponding to $S$.

\begin{defn}\label{D_inf count and multiplicity}
    Let $f\colon  [n] \to [k]$ be given. We define the \emph{$i$-th count} of $f$ as 
    \[
    c_f(i) := |f^{-1}(i)|
    \] 
    and the \emph{multiplicity} of $f$ at $i$ as 
    \[
    m_f(i) = \sum_{j \in f^{-1}(i)}{\prod_{t=1}^{j-1}{b_{f(t)}}}.
    \]
\end{defn}

\begin{rem}
Given $w\in S^n$, there exists a (not necessarily unique) function $f\colon [n]\to [k]$ such that $w=s^{f(1)}\cdots s^{f(n)}$. Then the $i$-th count of $f$ is equal to the total number of times that $s^{(i)}$ appears in the above presentation; and the multiplicity of $f$ at $i$ is the signed number of times that $a_i$ appears in the first component.
\end{rem}

\begin{exmpl}
Let $S = \{s^{(1)} = (1, 1), s^{(2)} = (10, 1), s^{(3)} = (100, -1)\}$ and consider the element $w \in S^4$ given by $w = s^{(1)}s^{(3)}s^{(1)}s^{(2)} = (1, 1)(100, -1)(1, 1)(10, 1) = (1 + 100 - 1 - 10, -1) = (90, -1)$. The counts and multiplicities of $f$ (here $f(1)=1,f(2)=3,f(3)=1,f(4)=2$) are given below.
\[
\begin{array}{c|c|c}
    i & c_f(i) & m_f(i) \\ 
    \hline 
    1 & 2 & 0 \\
    2 & 1 & -1 \\ 
    3 & 1 & 1 
\end{array}
\]
\end{exmpl}

\begin{prop}\label{prop:mult-count}
    For all $w = s^{f(1)}\cdots s^{f(n)} \in S^n$:
   \begin{enumerate}
   \item For all $i \in [k]$, we have $|m_f(i)| \leq c_f(i)$
         \item $\sum_{i=1}^{k}c_f(i) = n$
         \item $m_f(i) \equiv c_f(i) \pmod{2}$; in particular, $\sum_{i=1}^k |m_f(i)| \equiv n \pmod{2}$. \label{same-parity}
     \end{enumerate}
\end{prop}

\begin{proof}
    The proofs of (1) and (2) are straightforward from the definitions. For (3), 
    \begin{equation*}
        c_f(i) - m_f(i) 
        = |f^{-1}(i)| - \sum_{j \in f^{-1}(i)}{\prod_{t=1}^{j-1}{b_{f(t)}}}
        = \sum_{j \in f^{-1}(i)}{\left(1 - \prod_{t=1}^{j-1}{b_{f(t)}}\right)}
    \end{equation*}
    We have each $b_i \in \{\pm 1\}$, so any product of them also lies in $\{\pm 1\}$, and hence $1 - \prod_{t=1}^{j-1}{b_{f(t)}} \in 1 - \{\pm 1\} = \{0, 2\}$ for all $j$. It follows that the right hand side of the above equation is a sum of $0$ and $2$ and thus even, proving the claim.
\end{proof}

It turns out that we can reconstruct $w$ out of the multiplicity data $\mathbf{m}_f = (m_f(1),\dots, m_f(k))$, which is thus independent of the choice of a function $f\colon [n]\to [k]$ performing $w=s^{f(1)}\cdots s^{f(n)}$. In other words, the assignment $S^n \rightarrow \mathbb{Z}^k$ carrying any $w\in S^n$ to $\mathbf{m}_f$ for an arbitrary $f\colon [n]\to [k]$ such that $w=s^{f(1)}\cdots s^{f(n)}$, is injective (notice that we do not claim it to be well-defined at the moment, so potentially, it could be one-to-many). More precisely, we have the following formula:

\begin{lem}\label{lem:uniqueRep}
    Let $w$ be given. Fix any $f:[n] \to [k]$ such that $w = s^{f(1)}\cdots s^{f(n)}$. Then
    \[
    w = \left(\sum_{i=1}^{k}m_f(i)a_i, \prod_{i=1}^{k}b_i^{m_f(i)}\right).
    \] 
\end{lem}

\begin{proof} Compute:
    \begin{align*}
        w &= (a_{f(1)}, b_{f(1)})(a_{f(2)},b_{f(2)})\cdots (a_{f(n)}, b_{f(n)})\\
        &= (a_{f(1)} + b_{f(1)}a_{f(2)}, b_{f(1)}b_{f(2)})(a_{f(3)}, b_{f(3)}) \cdots (a_{f(n)}, b_{f(n)})\\
        & \;\;\vdots\\
        &= \left(\sum_{j=1}^{n} \left(\prod_{t=1}^{j-1}b_{f(t)}\right)a_{f(j)}, \  \prod_{j=1}^{n}b_{f(j)}\right)\\
        &= \left(\sum_{i=1}^{k}\left(\sum_{j \in f^{-1}(i)}\prod_{t=1}^{j-1}b_{f(t)}\right)a_{i}, \prod_{i=1}^{k}b_i^{|f^{-1}(i)|}\right) & \textit{grouping by $s^{(i)}$ term}\\
        &= \left(\sum_{i=1}^{k}m_f(i)a_{i}, \prod_{i=1}^{k}b_i^{c_f(i)}\right) & \textit{by Definition \ref{D_inf count and multiplicity}}\\
        &= \left(\sum_{i=1}^{k}m_f(i)a_{i}, \prod_{i=1}^{k}b_i^{m_f(i)}\right) & \textit{by Proposition \ref{prop:mult-count} (\ref{same-parity})}
    \end{align*}
as claimed
\end{proof}

This shows that each $w \in S^n$ can be constructed from the $\mathbf{m}_f$. To tackle the question of well-definition, we require the following result.

\begin{prop}\label{prop:sparsity}
Given $k,n \in \mathbb{N}$, there exists a set $S_{k,n} = \{a_1,\dots,a_k\} \subset \mathbb{N}$ such that if $(u_1, \dots, u_k)$,$(v_1, \dots, v_k) \in \mathbb{Z}^k$ are such that $\sum_{i=1}^{k}|u_i|, \sum_{i=1}^{k}|v_i| \leq n$, and $\sum_{i=1}^k a_iu_i=\sum_{i=1}^{k} a_i v_i$, then $(u_1,\dots,u_k)=(v_1,\dots,v_k)$. Moreover, this holds for a generic $(a_1,\dots,a_k)$, namely, from a non-empty Zariski open subset of $\mathbb{Z}^k$.
\end{prop}

\begin{proof}
    It is equivalent to prove that for any given $N$, there is $\vec{a}\in \mathbb{Z}^k$ such that for any $0<\|\vec{u}\|_{1} \leq N$ we have $\vec{a} \cdot \vec{u}\neq 0$. Indeed, for any fixed $\vec{u}$, the set of all $\vec{a}\in \mathbb{Z}^k$ such that $\vec{a} \cdot \vec{u}=0$ is a hyperplane; since there are finitely many $\vec{u}\in \mathbb{Z}^k$ of $l^1$-norm bounded by $N$, we may take $\vec{a}$ an arbitrary vector in the complement for the union of these finitely many hyperplanes, a non-empty Zariski open subset of $\mathbb{Z}^k$.
\end{proof}

\section{Reflections and translations}
Let $k,n\in \mathbb{N}$ and $1\leq p\leq k$ be given.     We construct a set $S_{k,p,n} \subseteq D_\infty$ as follows; let $\{a_1, \dots, a_k\}$ be from $S_{k,n}$ as in Proposition \ref{prop:sparsity} (namely, generic) and $b_1=\cdots=b_p=-1,\ b_{p+1}=\cdots=b_k=1$, and denote $s^{(i)}=(a_i,b_i)$, so
    \begin{align} \label{dinfty set}
    S_{k,p,n} 
    &:= \{s^{(1)},\dots,s^{(k)}\} = \{(a_1, b_1), \dots, (a_k, b_k)\} \\
    & = \{(a_1, -1), (a_2, -1), \dots, (a_p, -1), (a_{p+1}, 1), \dots, (a_k, 1)\} \nonumber
    \end{align}
    i.e., we take $p$ reflections and $k-p$ translations.

Define $\phi_{k,p,n}: S^n_{k,p,n} \to \mathbb{Z}^k$ via 
\[ w \mapsto (m_f(1), \dots, m_f(k)) \] where $f$ is an arbitrary function $[n] \to [k]$ such that $w = s^{(f(1))} \cdots s^{(f(n))}$.

\begin{lem} \label{lem:rep}
    The map $\phi_{k,p,n}: S_{k,p,n}^n \to \mathbb{Z}^k$ is well-defined and injective.
\end{lem}

\begin{proof}
Suppose that $s^{(f(1))}\dots s^{(f(n))}=s^{(g(1))}\cdots s^{(g(n))}$. By Lemma \ref{lem:uniqueRep} we have $\sum_{i=1}^{k} m_f(i) a_i = \sum_{i=1}^{k} m_g(i) a_i$ for all $1\leq i\leq k$. By Proposition \ref{prop:sparsity}, we have that $\textbf{m}_f=\textbf{m}_g$. Thus $\phi_{k,p,n}$ is well-defined.    

Injectivity follows from Lemma \ref{lem:uniqueRep}, that shows how to compute $w=s^{(f(1)}\cdots s^{(f(n))}$ out of $\textbf{m}_f$.
\end{proof}

Hence for $w\in S_{k,p,n}^n$ we can write $m_w(i)=m_f(i)$ for an arbitrary $f\colon [n]\to [k]$ performing $w=s^{(f(1))}\cdots s^{(f(n))}$.

\begin{lem}\label{lem:tkn_all_you_need}
    \[
    \gamma_{D_\infty}(k,p,n) = |S_{k,p,n}^n|
    \]
\end{lem}

\begin{proof}
    Let $S = \{(\alpha_1, \beta_1), \dots, (\alpha_k, \beta_k)\} \subset D_\infty$ be given. Suppose that $S$ contains $p$ reflections; we may assume that these are the first $p$ ones without loss of generality (we may readily assume that $p>0$ as otherwise $S$ is contained in $\mathbb{Z}\hookrightarrow D_\infty$).
    Recall that, from Lemma \ref{lem:uniqueRep},
    \[
    S^n = \left\{\left(\sum_{i=1}^{k}m_f(i) \alpha_i, \prod_{i=1}^{k} \beta_i^{m_f(i)}\right) \ \Big|\  f:[n]\to[k]\right\}.
    \]
By Lemma \ref{lem:rep} and Lemma \ref{lem:uniqueRep}, the assignment $S_{k,p,n}^n\to S^n$ carrying $$\nu \colon s^{(f(1))}\cdots s^{(f(n))}\mapsto (\alpha_{f(1)},\beta_{f(1)})\cdots (\alpha_{f(n)},\beta_{f(n)})$$ is well-defined. 
(Notice that the orientation of the $i$-th element in $S$ is identical to that of $S_{k,p,n}$.) 
Clearly $\nu$ is surjective, so $|S^n|\leq~|S_{k,p,n}^n|$.
\end{proof}

Our next goal is to compute $|\Im \phi_{k,p,n}|$ and, subsequently, determine the optimal $1\leq p\leq k$ that attains $\gamma^{\max}_{D_\infty}(k,n)$. (Notice that $|\Im \phi_{k,0,n}|=|\mathcal{C}(n,k)|={n+k-1 \choose k-1}$.)

\begin{prop}\label{proper:pos-neg-mult-sum}
Let $k,n$ and $1\leq p\leq k$ be given.
Suppose that $w \in S_{k,p,n}^n$ has $\sum_{i=1}^{p}|m_w(i)| = q$. Then
$$ \sum_{\substack{1\leq i\leq p:\\ m_w(i) > 0}}|m_w(i)| = \lceil \frac{q}{2} \rceil\ \ \text{and} \ \sum_{\substack{1\leq i\leq p:\\ m_w(i) \leq 0}} |m_w(i)| = \lfloor \frac{q}{2} \rfloor.$$
\end{prop}

\begin{proof}
Let $w = s^{(f(1))}\cdots s^{(f(n))}$ for some $f \colon [n] \to [k]$. Then, for any $s^{(i)}$, we have
    \begin{equation} \label{eq:multiplicity}
    m_w(i) = m_f(i) = \sum_{j \in f^{-1}(i)}\prod_{m=1}^{j-1}b_{f(m)} = \sum_{j \in f^{-1}(i)}(-1)^{\#\{m \in f^{-1}([p]) : \  m < j\}}
    \end{equation}
    where the final step removes the contributions $b_i = 1$ (coming from the translations $s^{(p+1)},\dots,s^{(k)}$). Denote $r = |f^{-1}([p])| = \sum_{i=1}^{p}c_f(i)$. Then,
    \begin{align}
    \sum_{i=1}^{p}m_w(i) 
    &= \sum_{i=1}^{p}\sum_{j \in f^{-1}(i)}(-1)^{\#\{m \in f^{-1}([p]) :\ m < j\}} \nonumber \\
    &= \sum_{j \in f^{-1}([p])}(-1)^{\#\{m \in f^{-1}([p]) :\ m < j\}} = \sum_{t=1}^{r}(-1)^{t-1} = \begin{cases}
        0, & \text{if}\ r\ \text{is even} \\
        1, & \text{if}\ r\ \text{is odd.}
    \end{cases}
    \end{align}
    Recall that $m_w(i) = m_f(i) \equiv c_f(i) \pmod{2}$ by Proposition \ref{prop:mult-count} (\ref{same-parity}). Therefore, we have $q \equiv r \pmod{2}$, so we can replace $r$ with $q$ in the previous expression. 

    Denote 
    \begin{align*}
        [p]^+ & =\{1\leq i\leq p\;|\;m_i>0\} \\ X & = \sum_{i\in [p]^+} |m_w(i)|, \ \ \ Y = \sum_{i\in [p]\setminus [p]^+} |m_w(i)|.
    \end{align*}
    Thus $X+Y=q$ and $X-Y=\sum_{i\in [p]} m_w(i)$ is $0$ if $q$ is even, and $1$ if $q$ is odd. It follows that $X=\lceil \frac{q}{2} \rceil,\ Y=\lfloor \frac{q}{2} \rfloor$, as claimed.
\end{proof}

We will freely use Equation \eqref{eq:multiplicity} throughout the rest of the section.

\begin{prop}\label{prop:zero-mult}
Let $k,n$ and $1\leq p\leq k$ be given. Let $w \in S_{k,p,n}^n$ be such that $\sum_{i=1}^{p}|m_w(i)| = 0$. Then, $\sum_{i=1}^{k}|m_w(i)| < n$ or $m_w(i) \geq 0$ for all $i \in \{p+1, \dots, k\}$.
\end{prop}

\begin{proof}
    Let $w = s^{(f(1))}\cdots s^{(f(n))}$ for some $f: [n] \to [k]$. Assume that $\sum_{i=1}^{p}|m_w(i)| = 0$, and $\sum_{i=1}^{k}|m_w(i)| = n$ (recall that $\sum_{i=1}^k |m_w(i)|\leq \sum_{i=1}^k c_f(i) = n$). By Proposition \ref{prop:mult-count}, we have
    \begin{equation*}
    n =\sum_{i=1}^{k}|m_w(i)| \leq c_f(1) + \sum_{i=2}^{k}|m_w(i)| \leq \dots \leq \sum_{i=1}^{k}c_f(i) = n
    \end{equation*}
    so we must have $c_f(i) = m_w(i)$ for all $i \in [k]$. Then, since $\sum_{i=1}^{p}|m_w(i)| = 0$, we have $m_w(i) = c_f(i) = 0$ for all $i \in [p]$. However, this tells us that $f^{-1}([p]) = \emptyset$, so, for all $p+1\leq j\leq k$,
    \[
        m_w(j) = \sum_{i \in f^{-1}(j)}(-1)^{\#\{m \in f^{-1}([p]) :\ m < i\}} = \sum_{i \in f^{-1}(j)}(-1)^0 \geq 0,
    \]
as required.
\end{proof}

\begin{prop}\label{prop:mult-to-word}
    Let $k,n$ and $1\leq p\leq k$ be given. A vector $\mathbf{m} = (m_1, \dots, m_k)~\in~\mathbb{Z}^k$ satisfying $$\underbrace{\sum_{i=1}^{p}|m_i|}_{=: q} \leq \underbrace{\sum_{i=1}^{k}|m_i|}_{=:l} \leq n$$ lies in $\Im(\phi_{k,p,n})$ if and only if 
   \begin{enumerate}
        \item $n \equiv l \pmod{2}$
        \item $\sum_{1\leq i\leq p:\ m_i > 0}|m_i| = \lceil q/2 \rceil$ and $\sum_{1\leq i\leq p:\ m_i \leq 0}|m_i| = \lfloor q/2 \rfloor$
        \item If $q = 0$, then either $l < n$ or $m_i \geq 0$ for all $i \in \{p+1, \dots k\}$.
   \end{enumerate}
\end{prop}

\begin{proof}
    Propositions \ref{prop:mult-count}, \ref{proper:pos-neg-mult-sum}, and \ref{prop:zero-mult} show that all $\mathbf{m} \in \Im(\phi_{k,p,n})$ satisfy the assertions, and it remains to prove the converse. 
    
Let $\textbf{m}\in \mathbb{Z}^k$ satisfying the assertions in the statement of the theorem be given. We aim to construct a function $f \colon [n] \to [k]$ such that $(m_f(1),\dots,m_f(k))=\textbf{m}$. Adopt the notation $x^{[n]}:=\underbrace{x,\dots,x}_{n\ \text{times}}$.

We start by considering the relative assignments of the reflection elements. Let $1\leq i_1<\cdots<i_t \leq p$ be the indices (from $[p]$) for which $m_{*} \geq 0$. Let 
$$R_+ := (i_1^{[m_{i_1}]},\dots,i_t^{[m_{i_t}]})$$
and denote by $R_+(j)$ the $j$-th entry in this tuple. Notice that the length of $R_+$ is $\sum_{1\leq i\leq p:\ m_i\geq 0} m_i = \lceil \frac{q}{2} \rceil$ by assumption.
Likewise, let $1\leq j_1<\cdots <j_s \leq p$ be the indices for which $m_{*} < 0$ and put 
$$R_{-} := (j_1^{[|m_{j_1}|]},\dots,j_s^{[|m_{j_s}|]})$$
and similarly, its length is $\sum_{1\leq i\leq p:\ m_i < 0} |m_i| = \lfloor \frac{q}{2} \rfloor$.

Define a function $g\colon [q] \to [p]$ as follows.
The function $g$ will alternate indices from $R_+$ and $R_-$. 
For all $1 \leq j \leq \lceil \frac{q}{2} \rceil$ let $g(2j-1) = R_+(j)$ and for $1 \leq j \leq \lfloor \frac{q}{2} \rfloor$ let $g(2j) = R_{-}(j)$. 

We now define our target function $f\colon [n] \to [k]$. 

\medskip

\noindent \emph{Case I: $q > 0$.} 
Let $p+1 \leq x_1 < \cdots < x_u \leq k$ be the indices for which $m_{*} \geq 0$, and let $p+1\leq y_1 < \cdots < y_v \leq k$ be the indices for which $m_{*} < 0$. Consider
$$(x_1^{[m_{x_1}]},\dots,x_u^{[m_{x_u}]},g(1),y_1^{[|m_{y_1}|]},\dots,y_v^{[|m_{y_v}|]},g(2),\dots,g(q),1^{[n-l]})$$
(recall that $l:=\sum_{i=1}^{k} |m_i|$.) Notice that the length of this tuple is $n$. We regard it as a function, which we denote $f$, reading $f(1),\dots,f(n)$ from left to right.

We claim that $(m_f(1),\dots,m_f(k))=(m_1,\dots,m_k)$. 
For $p+1\leq z\leq k$, if $m_z \geq 0$, notice that $z=x_r$ for some $1\leq r\leq u$ and thus appears totally $m_z$ times, all placed before the first occurrence of any $g(1),\dots,g(q)$, namely, before any $[p]$ has appeared. Hence $$m_f(z)=\sum_{j\in f^{-1}(z)} (-1)^0 = |f^{-1}(z)|=m_z.$$ If $m_z < 0$ then $z=y_r$ for some $1\leq r \leq v$ and thus $z$ appears totally $|m_z|$ times, all placed between the first and second occurrence of an index from $[p]$. Hence $$m_f(z)=\sum_{j\in f^{-1}(z)} (-1)^1 = -|f^{-1}(z)|=-|m_z|=m_z.$$ 

We now consider $1 < z \leq p$; exclude the case $z=1$ for the moment. Suppose first $m_z \geq 0$. Then $z=i_r$ for some $1\leq r\leq t$; recall that $z$ appears $m_z$ consecutive times in $R_+$. Since $g(2j-1)=R_+(j),g(2j)=R_-(j)$, it follows that every occurrence of $z$ in $f$ appears after an \emph{even} number of appearances of indices from $[p]$:
\begin{align*}
m_f(z) & = \sum_{\mu\in f^{-1}(z)} (-1)^{|\{m\in f^{-1}([p]):\ m<\mu\}|} \\ & = \sum_{j:\ R_+(j)=z} (-1)^{(2j-1)-1} =|\{j:\ R_+(j)=z\}|=m_z.
\end{align*}
Likewise, if $m_z < 0$ then $z=j_r$ for some $1\leq r \leq s$, and $z$ appears $|m_z|$ consecutive times in $R_{-}$. Since $g(2j-1)=R_+(j),g(2j-1)=R_{-}(j)$, it follows that every occurrence of $z$ in $f$ appears after an \emph{odd} number of appearances of indices from~$[p]$:
\begin{align*}
m_f(z) & = \sum_{\mu\in f^{-1}(z)} (-1)^{|\{m\in f^{-1}([p]):\ m<\mu\}|} \\ & = \sum_{j:\ R_{-}(j)=z} (-1)^{(2j)-1} =-|\{j:\ R_{-}(j)=z\}|=-|m_z|=m_z.
\end{align*}
To finish Case I, it remains to examine $z=1$. The difference from other cases of $1\leq z\leq p$ is that $1$ appears (in addition to the appearances coming from $g(1),\dots,g(q)$ in the first $l$ indices in $f$) $n-l$ more times in the end of $f$. But the contribution of the last $n-l$ appearances of $1$ to $m_f(1)$ is $(-1)^q+(-1)^{q+1}+\cdots+(-1)^{q+n-l-1} = \frac{(-1)^{q+n-l} - (-1)^q}{-2} = 0$ since $n \equiv l \pmod{2}$. 
Hence $\textbf{m}_f = \textbf{m}$.

\medskip

\noindent \emph{Case II: $q=0$.}
In this case, we either have $l < n$, or $m_i \geq 0$ for all $i \in \{p+1, \dots, k\}$ by assertion (3) in the statement of the theorem. As before, let $p+1 \leq x_1 < \cdots < x_u \leq k$ be the indices for which $m_{*} \geq 0$, and let $p+1\leq y_1 < \cdots < y_v \leq k$ be the indices for which $m_{*} < 0$. We let $(f(1),\dots,f(n))$ be
$$(x_1^{[m_{x_1}]},\dots,x_u^{[m_{x_u}]},1,y_1^{[|m_{y_1}|]},\dots y_v^{[|m_{y_v}|]},1^{[n-l-1]}).$$
Notice that this is well-defined if $l<n$; and if $l=n$ then by assumption, for all $p+1\leq z\leq k$ we have $m_z\geq 0$ and hence the above can be viewed as a `truncated' vector $(x_1^{[m_{x_1}]},\dots,x_u^{[m_{x_u}]})$. 
Now for any $2\leq z\leq p$, we have $m_f(z)=0$ as $z$ never appears in $f$; indeed, $m_z=0$ (since $q=0$). For $z=1$, we have a similar analysis to Case I: $$m_f(z)=(-1)^0+(-1)^1+\cdots+(-1)^{n-l-1}=\frac{(-1)^{n-l} - 1}{-2}=0$$ as $n \equiv l \pmod{2}$. That $m_f(z)=m_z$ for all $p+1\leq z\leq k$ follows as in Case I.
\end{proof}

\section{Growth of subsets with a given number of reflections}
We now focus on finding a closed-form expression for the number of those $\mathbf{m} \in \mathbb{Z}^k$ satisfying the criterion given in Proposition \ref{prop:mult-to-word}, which is equivalent to counting $|S_{k,p,n}^n|$. Recall that the case $p=0$ corresponds to sparse subsets of $\mathbb{Z}$, and indeed $|\Im \phi_{k,0,n}|=|S_{k,0,n}^n|={n+k-1 \choose k-1}=\gamma_\mathbb{Z}^{\max} (k,n)$. 
It remains to consider the case of a positive number of reflections, $p$.

Let $x,y\in \mathbb{Z}_{\geq 0}$. Let $N(x,y)$ denote the number of different ways to assign integer multiplicities to $x$ items, such that the absolute values of these multiplicities add up to $y$.

\begin{prop} \label{prop:N}
We have for $x,y > 0$
\begin{align*}
    N(x,y) = 
    \sum_{s=1}^{\min(x,y)}2^s\binom{x}{s}\binom{y - 1}{s-1}.
\end{align*}
and $N(x, 0) = 1$, $N(0, y) = \mathbbm{1}_{y=0}$. Moreover,
$$    G_x(T) := \sum_{y \geq 0}N(x, y)T^y = \left(\frac{1+T}{1-T}\right)^x,\ \ \text{and}\ \ \sum_{x,y\geq 0} N(x,y) X^x Y^y = \frac{1-Y}{1-X-Y-XY}.$$
\end{prop}

\begin{proof}
    Let $s$ be the number of items with non-zero multiplicity. If $y = 0$, no such elements exist, so $N(x,y) = 1$ with all multiplicities equal to 0. If $x = 0$, then the sum of multiplicities $y$ is necessarily equal to 0. Otherwise, $s$ ranges from $1$ to $\min(x,y)$. There are $\binom{x}{s}$ ways to choose these $s$ elements, $2^s$ ways to assign signs, and $\binom{y-1}{s-1}$ ways to distribute multiplicities such that their absolute values sum to $y$ (distributing $y$ balls into $s$ bins such that they are all non-empty).

\medskip

Next, for the computation of $G_x(T)$ (for any fixed $x$), notice that 
\[
        N(x,y) = \sum_{\substack{(p_1, \dots, p_k) \in \mathbb{Z}_{\geq 0}^k \\ p_1+\cdots+p_x=y}}\prod_{i=1}^x N(1, p_i).
\]
Moreover, $N(1,p_i)=1 + \mathbbm{1}_{\{p_i>0\}}$ (indeed, if $p_i>0$ we can assign both $\pm p_i$). Subsequently,
\begin{align*}
\sum_{y\geq 0} N(x,y)T^y = \left(1+2T+2T^2+2T^3+\cdots\right)^x = \left(1+\frac{2T}{1-T}\right)^x=\left(\frac{1+T}{1-T}\right)^x
\end{align*}
which is convergent iff $|T|<1$. 
Finally,
\begin{align*}
\sum_{x,y\geq 0} N(x,y) X^x Y^y & = \sum_{x\geq 0}\left(\sum_{y\geq 0} N(x,y)Y^y \right) X^x = \sum_{x\geq 0}\left(\frac{1+Y}{1-Y}\right)^xX^x \\ & = \frac{1}{1-\frac{1+Y}{1-Y}X} = \frac{1-Y}{1-X-Y-XY}
\end{align*}
as claimed.
\end{proof}

Let $x,y\geq 1$ be integers. Let $R(x, y)$ be the number of different ways to assign integer multiplicities to $x$ items, such that the absolute values of these multiplicities add up to $y \geq 1$, and such that
\begin{enumerate}
    \item The positive multiplicities sum to $\ceil{y/2}$
    \item The negative multiplicities sum to $-\floor{y/2}$.
\end{enumerate}

\begin{prop} \label{prop:R}
If $x,y>1$
\begin{align*}
    R(x,y) = 
    \sum_{r=1}^{\min(x-1,\floor{y/2})}{\binom{x}{r}\binom{\floor{y/2} - 1}{r-1}\binom{\ceil{y/2} + (x-r) - 1}{(x-r) - 1}}
\end{align*}
and $R(x,1)=x$, $R(1,y) = \mathbbm{1}_{y=1}.$
\end{prop}
\begin{proof}
For $y = 1$, there is exactly one item of multiplicity $1$, for which we have $x$ options, and the rest of the items must be assigned with multiplicity zero. For $x = 1$, this singular element must have multiplicity $1$ to satisfy the parity conditions, so there is exactly one assignment iff $y = 1$. If $x,y \geq 2$, let $r$ be the number of elements with negative multiplicity, ranging from $1$ to $\min(x-1, \floor{y/2})$. There are $\binom{x}{r}$ ways to choose these $r$ elements, $\binom{\floor{y/2} - 1}{r-1}$ ways of assigning negative multiplicities to these $r$ elements such that they sum to $-\floor{y/2}$, and $\binom{\ceil{y/2} + (x-r) - 1}{(x-r) - 1}$ ways of assigning non-negative multiplicities to the remaining $x-r$ elements such that they sum to $\ceil{y/2}$.
\end{proof}

We can now give a closed formula for $|S_{k,p,n}^n|$ for all $1 \leq p\leq k$.

\begin{proof}[{Proof of Theorem \ref{thm:k p n}}]
We need to prove that for every $1\leq p\leq k$ and $n$
\begin{align*}
        |S_{k,p,n}^n| 
        = \binom{n + k - p - 1}{k - p - 1} & +  \sum_{\substack{0 \leq l < n \\ n \equiv l \bmod{2}}} N(k-p,l) + \sum_{\substack{1 \leq l \leq n \\ n \equiv l \bmod{2}}}\sum_{q=1}^{l}R(p,q)N(k-p,l-q).
    \end{align*}
By Lemma \ref{lem:rep}, it is equivalent to compute $|\Im(\phi_{k,p,n})|$, which we approach using Proposition \ref{prop:mult-to-word}. We think of a vector $\textbf{m}\in \Im(\phi_{k,p,n})$ as an assignment of integral multiplicities $m_1,\dots,m_k$ to the elements $(a_1,b_1),\dots,(a_k,b_k)$. We split based on the value of $q=\sum_{i=1}^p |m_i|$ (as defined in Proposition \ref{prop:mult-to-word}); recall that this is the sum of the absolute values of the multiplicities assigned with the reflections. 

If $q = 0$, then the multiplicities of the reflections must all be $0$. Therefore, we must assign multiplicities to the $k-p$ translations such that their absolute values sum to $l$ for all $0\leq l\leq n$ such that $l\equiv n \pmod{2}$ and either (i) $l<n$ or (ii) if $l = n$, then all multiplicities of the translations are non-negative. 
The vectors $\textbf{m}$ falling in case (i), for each given $l$, are in bijection with vectors $(m_{p+1},\dots,m_k)\in \mathbb{Z}^{k-p}$ such that $\sum_{i=p+1}^{k} |m_i|=l$; this set is of cardinality $N(k-p,l)$ by definition, and hence the total number of vectors falling in case (i) is $$\sum_{\substack{0 \leq l < n \\ n \equiv l \bmod{2}}} N(k-p,l).$$ 
The number of vectors $\textbf{m}$ falling in case (ii) is just the number of partitions of $n$ balls into $k-p$ bins, that is, $$\binom{n + k - p - 1}{k - p - 1}.$$

We now turn to the case where $q > 0$. In that case, we must have $1\leq q\leq l\leq n$ and $l \equiv n \pmod{2}$. Fix such $l$; we let $q$ range from $1$ to $l$. Given such $q,l$, the number of vectors $\textbf{m}=(m_1,\dots,m_k)$ satisfying $\sum_{i=1}^p |m_i|=q$ and $\sum_{i=1}^k |m_i|=l$ (or, equivalently, $\sum_{i=p+1}^k |m_i|=l-q$) and adhering to Condition (2) in Proposition \ref{prop:mult-to-word} is $R(p,q)$ (the number of options for $(m_1,\dots,m_p)$) times $N(k-p,l-q)$ (the number of options for $(m_{p+1},\dots,m_k)$). Altogether, the number of vectors with $q>0$ is
$$\sum_{\substack{1 \leq l \leq n \\ n \equiv l \bmod{2}}}\sum_{q=1}^{l}R(p,q)N(k-p,l-q).$$
Collecting pieces, the claim is proved.
\end{proof}

\begin{cor} \label{cor:1 k n}
    We have $$\gamma_{D_\infty}(k,1,n) =|S_{k,1,n}^n| =1 + \binom{n + k - 2}{k - 2} + \sum_{s=1}^{\min(k-1,n-1)}2^s\binom{k-1}{s}\binom{n-1}{s}$$
    for all $k,n$.
\end{cor}
\begin{proof}
    By Theorem \ref{thm:k p n},
    \begin{align*}
        |S_{k, 1, n}^n|
        &= \binom{n + k - 2}{k - 2} +  \sum_{\substack{0 \leq l < n \\ n \equiv l \bmod{2}}} N(k-1,l) + \sum_{\substack{1 \leq l \leq n \\ n \equiv l \bmod{2}}}N(k-1,l-1) \\
        &= \binom{n + k - 2}{k - 2} +  \sum_{l=0}^{n-1} N(k-1,l)\\
        &= 1 + \binom{n + k - 2}{k - 2} +  \sum_{l=1}^{n-1}\sum_{s=1}^{\min(k-1,l)}2^s\binom{k-1}{s}\binom{l - 1}{s-1}\\
        &= 1 + \binom{n + k - 2}{k - 2} + \sum_{s=1}^{\min(k-1,n-1)}2^s\binom{k-1}{s}\sum_{l=s}^{n-1}\binom{l - 1}{s-1}\\
        &= 1 + \binom{n + k - 2}{k - 2} + \sum_{s=1}^{\min(k-1,n-1)}2^s\binom{k-1}{s}\binom{n-1}{s}.
    \end{align*}

    The first step follows from the fact that $R(1,q) = \mathbbm{1}_{\{q=1\}}$. The final step applies the hockey-stick identity to the inner summation.
\end{proof}

\section{Growth of subsets in the infinite dihedral group}

\begin{lem}\label{lem:N-asymp}
Fix a positive integer $x$. Then
\[
    N(x,y) = \frac{2^{x}}{(x-1)!}y^{x-1} + O(y^{x-2}).
\]
\end{lem}
\begin{proof}
By Proposition \ref{prop:N}, for $y\geq x$,
    \begin{align*}
        N(x,y) & =  \sum_{s=1}^{\min(x,y)} 2^s \binom{x}{s} \binom{y-1}{s-1} \\ & =\sum_{s=1}^{x}2^s\binom{x}{s}\left(\frac{y^{s-1}}{(s-1)!} + O(y^{s-2})\right) = \frac{2^{x}}{(x-1)!}y^{x-1} + O(y^{x-2}).
    \end{align*}
\end{proof}

\begin{lem}\label{lem:R-asymp}
Fix an integer $x > 1$. Then
\[
    R(x,y) = \frac{2^{2-x}}{(x-2)!}\binom{2x-2}{x-1}y^{x-2} + O(y^{x-3}).
\]
\end{lem}

\begin{proof}
By Proposition \ref{prop:R}, for $y\geq 2x-2$,
\begin{align*}
    R(x,y) 
    &= 
    \sum_{r=1}^{\min(x-1,\floor{y/2})}{\binom{x}{r}\binom{\floor{y/2} - 1}{r-1}\binom{\ceil{y/2} + (x-r) - 1}{(x-r) - 1}} \\ & = \sum_{r=1}^{x-1}{\binom{x}{r}\binom{\floor{\frac{y}{2}} - 1}{r-1}\binom{\ceil{\frac{y}{2}} + (x-r) - 1}{(x-r) - 1}}\\
    &= \sum_{r=1}^{x-1}{\binom{x}{r}\left(\frac{(\frac{y}{2})^{r-1}}{(r-1)!} + O(y^{r-2})\right)\left(\frac{(\frac{y}{2})^{x-r-1}}{(x-r-1)!} + O(y^{x-r-2})\right)}\\ 
    &= \sum_{r=1}^{x-1}{\binom{x}{r}\left(\frac{1}{(r-1)!(x-r-1)!}\left(\frac{y}{2}\right)^{x-2} + O(y^{x-3})\right)}\\
    &= C_xy^{x-2} + O(y^{x-3})
\end{align*}
where
\begin{align*}
    C_x 
    & = 2^{2-x}\sum_{r=1}^{x-1}\binom{x}{r}\frac{1}{(r-1)!(x-r-1)!} 
   \\ & = \frac{2^{2-x}}{(x-2)!}\sum_{r=1}^{x-1}\binom{x}{r}\binom{x-2}{x-r-1} \\ & = \frac{2^{2-x}}{(x-2)!}\sum_{r=0}^{x-1}\binom{x}{r}\binom{x-2}{x-r-1} = \frac{2^{2-x}}{(x-2)!}\binom{2x-2}{x-1}
\end{align*}
where the final equality follows from Vandermonde's identity (that is, $\sum_{i=0}^k \binom{m}{i} \binom {n}{k-i} = \binom{m+n}{k}$). The bound on the remainder follows similarly to \ref{lem:N-asymp}.
\end{proof}

\begin{proof}[{Proof of Theorem \ref{thm:k fixed}}]
Let $k,p,n \in \mathbb{N}$ with $1 \leq p \leq k$ fixed. We aim to prove
    \[
    |S_{k,p,n}^n| = \frac{2^{k-1}}{(k-1)!}\frac{\binom{2p-2}{p-1}}{2^{2p-2}}n^{k-1} + O(n^{k-2}).
    \]
Recall, from Theorem \ref{thm:k p n}, that
    \begin{align} \label{formula}
        |S_{k,p,n}^n| 
        &= \underbrace{\binom{n + k - p - 1}{k - p - 1}}_{(I)} +  \underbrace{\sum_{\substack{0 \leq l < n \\ n \equiv l \bmod{2}}} N(k-p,l)}_{(II)} + \underbrace{\sum_{\substack{1 \leq l \leq n \\ n \equiv l \bmod{2}}}\sum_{q=1}^{l}R(p,q)N(k-p,l-q)}_{(III)}
    \end{align}
Let $\alpha(x)$ and $\beta(x)$ be the leading coefficients of $N(x,y)$ and $R(x,y)$, respectively, viewed as functions of $y$. By Lemmas \ref{lem:N-asymp} and \ref{lem:R-asymp}, 
$$\alpha(x) = \frac{2^x}{(x-1)!},\ \ \beta(x) = \frac{2^{2-x}}{(x-2)!} \binom{2x-2}{x-1}.$$
There are three summands in (\ref{formula}), which we denote $(I),(II),(III)$. The first two are easy to analyze given our constraints, yielding
\[
  (I) =  \binom{n + k - p - 1}{k - p - 1} = \frac{1}{(k-p-1)!}n^{k-p-1} + O(n^{k-p-2})
\]
if $1 \leq p \leq k-1$ and $(I) = 0$ if $p=k$. In any case that $p\geq 1$, we have $(I)=O(n^{k-2})$, which will turn out to be negligible compared to $(II),(III)$.

\medskip

By Lemma \ref{lem:N-asymp}, when $1\leq p\leq k-1$ and $l > 0$, we can write $N(k-p,l)=\alpha(k-p)\cdot l^{k-p-1} + \Delta_l$ where $|\Delta_l|\leq K\cdot l^{k-p-2}$ for some constant $K$ depending solely on $k$ (but not on $l$). Thus
\begin{align*}
    (II) & = \sum_{\substack{0 \leq l < n \\ n \equiv l \bmod{2}}} N(k-p,l) = \mathbbm{1}_{\{n\ \text{even}\}} + \sum_{\substack{1 \leq l \leq n \\ n \equiv l \bmod{2}}}\alpha(k-p) l^{k-p-1} + \Delta_l \\ & = \mathbbm{1}_{\{n\ \text{even}\}} +\alpha(k-p) \cdot \sum_{\substack{1 \leq l \leq n \\ n \equiv l \bmod{2}}} l^{k-p-1} + \sum_{\substack{1 \leq l \leq n \\ n \equiv l \bmod{2}}} \Delta_l \\ & = \frac{\alpha(k-p)}{2}\sum_{l=1}^{n}l^{k-p-1} + O(n^{k-p-1}) \\ & = \frac{\alpha(k-p)}{2} \cdot \frac{n^{k-p}}{k-p} + O(n^{k-p-1}) =  \frac{2^{k-p-1}}{(k-p)!}n^{k-p} + O(n^{k-p-1}).
\end{align*}
Now if $2\leq p\leq k-1$ we see that $(II)=O(n^{k-2})$; if $p=1$ then $$(II)=\frac{2^{k-2}}{(k-1)!} n^{k-1} + O(n^{k-2})=\Theta(n^{k-1}).$$ If $p=k$ then we get $$(II) = \sum_{\substack{0 \leq l < n \\ n \equiv l \bmod{2}}} N(0,l)=\mathbbm{1}_{\{n\ \text{even}\}}.$$

The third summand $(III)$ is more involved. We first tackle the inner summation for each given $1\leq l\leq n,\ l\equiv n \pmod{2}$.
We now write for brevity $\alpha=\alpha(k-p),\beta=\beta(p)$. 

Consider the case $2 \leq p \leq k-1$. Recall that by Lemmas \ref{lem:N-asymp} and \ref{lem:R-asymp}, we have some constant $K$ depending only on $k,p$ (not on $q,l$) such that
\begin{eqnarray*}
R(p,q) & =& \beta q^{p-2} + \Delta_{q}; \\ N(k-p,l-q) & = &\alpha (l-q)^{k-p-1}+E_{l,q}\ (\text{if}\ q \neq l);\ \text{and} \\ |\Delta_q| & \leq &  Kq^{p-3}, \ \ \ |E_{l,q}|\leq K(l-q)^{k-p-2}.
\end{eqnarray*}

In the case that $q = l$, we can follow the analysis of (II) above to get the negligible contribution (for $p < k$)
\[
\sum_{\substack{1 \leq l \leq n \\ n \equiv l \bmod{2}}}R(p, l)N(k-p,0) = \sum_{\substack{1 \leq l \leq n \\ n \equiv l \bmod{2}}}R(p, l) = \frac{\beta}{2}\cdot\frac{n^{p-1}}{p-1} + O(n^{p-2}).
\]
Otherwise, returning to the inner summation, we have
\begin{align*}
\Big|R(p,q)N(k-p,l-q) - \alpha \beta q^{p-2} (l-q)^{k-p-1} \Big| & = \Big|\beta q^{p-2}E_{l,q} + \alpha(l-q)^{k-p-1} \Delta_q + \Delta_q E_{l,q} \Big| \\ & \leq M\left(q^{p-2}(l-q)^{k-p-2} + q^{p-3}(l-q)^{k-p-1}\right)
\end{align*}
for some constant $M$ depending solely on $k,p$ but not on $l,q$. Therefore,
\begin{align} \label{sum RN}
\Big|\sum_{q=1}^{l-1} R(p,q)N(k-p,l-q) - \alpha \beta \sum_{q=1}^{l-1} q^{p-2} (l-q)^{k-p-1} \Big| & \leq M\sum_{q=1}^{l-1} q^{p-2}(l-q)^{k-p-2} \nonumber \\ & + M\sum_{q=1}^{l-1} q^{p-3}(l-q)^{k-p-1}
\end{align}
Now
\begin{align*}
    \sum_{q=1}^{l-1} q^a(l-q)^b = \sum_{q=1}^l q^a(l-q)^b = l^{a+b+1}\cdot \frac{1}{l}\sum_{q=1}^l \left(\frac{q}{l}\right)^a \left(1-\frac{q}{l}\right)^b
\end{align*}
and
\begin{align*}
    \Big|\frac{1}{l}\sum_{q=1}^l \left(\frac{q}{l}\right)^a \left(1-\frac{q}{l}\right)^b-\int_0^1 \underbrace{x^a (1-x)^b}_{=:u_{a,b}(x)} dx \Big| \leq \frac{\max_{x\in [0,1]} |u_{a,b}'(x)|}{l}\leq \frac{a+b}{l}
\end{align*}
Recall that $\int_0^1 u_{a,b}(x) dx=\frac{\Gamma(a)\Gamma(b)}{\Gamma(a+b+1)}$, so
\begin{align*}
    \Big| \sum_{q=1}^l q^a(l-q)^b - \frac{a!b!}{(a+b+1)!} l^{a+b+1} \Big| \leq (a+b) l^{a+b}.
\end{align*}
Back to (\ref{sum RN}),
\begin{align}
\Big|\sum_{q=1}^{l-1} R(p,q)N(k-p,l-q) - \alpha \beta \frac{(p-2)!(k-p-1)!}{(k-2)!} l^{k-2} \Big| \leq Cl^{k-3}
\end{align}
for some $C$ depending solely on $k,p$ but not on $l$.
Now

\begin{align*}
\sum_{\substack{1\leq l\leq n \\ n\equiv l \pmod{2}}}\sum_{q=1}^{l-1} R(p,q)N(k-p,l-q) & = \alpha \beta \frac{(p-2)!(k-p-1)!}{(k-2)!} \sum_{\substack{1\leq l\leq n \\ n\equiv l \pmod{2}}}  l^{k-2} + O(n^{k-2}) \\ & = \alpha \beta \frac{(p-2)!(k-p-1)!}{(k-2)!}\cdot \frac{1}{2(k-1)}n^{k-1} + O(n^{k-2})
\end{align*}
The coefficient of $n^{k-1}$ in the above expression is
\begin{align}
& \frac{2^{k-p}}{(k-p-1)!} \cdot \frac{2^{2-p}}{(p-2)!} \binom{2p-2}{p-1} \cdot \frac{(p-2)!(k-p-1)!}{(k-2)!}\cdot \frac{1}{2(k-1)} \nonumber \\ & = \frac{2^{k-2p+1}}{(k-1)!}\binom{2p-2}{p-1}.
\end{align}
As established above, for $2\leq p\leq k-1$, we have $(I),(II)=O(n^{k-2})$, yielding the desired result.

Now, consider $p=1$. Then still $(I)=O(n^{k-2})$, but $(II)=\frac{2^{k-2}}{(k-1)!} n^{k-1} + O(n^{k-2})$. The above analysis of $(III)$ applies only to $2\leq p\leq k-1$, and for $p=1$ we have $R(p,q)=\mathbbm{1}_{\{q=1\}}$. Thus $\sum_{q=1}^l R(1,q)N(k-1,l-q)=N(k-1,l-1)$ and 
\begin{eqnarray*}
(III) & = & \sum_{\substack{1\leq l\leq n \\ n\equiv l \pmod{2}}}\sum_{q=1}^l R(1,q)N(k-1,l-q) \\ & = & \sum_{\substack{1\leq l\leq n \\ n\equiv l \pmod{2}}} N(k-1,l-1) = \frac{2^{k-2}}{(k-1)!} n^{k-1} + O(n^{k-2})    
\end{eqnarray*} 
where the last equality follows as in the analysis of $(II)$ above. Combining the components, we have that for $p=1$
\begin{align*}
(I)+(II)+(III) & = O(n^{k-2}) + \left(\frac{2^{k-2}}{(k-1)!} n^{k-1} + O(n^{k-2}) \right) + \left(\frac{2^{k-2}}{(k-1)!} n^{k-1} + O(n^{k-2})\right) \\ & =\frac{2^{k-1}}{(k-1)!} n^{k-1} + O(n^{k-2}).
\end{align*}
Finally, for $p=k$ we have $(I)=0$ and $(II)=\mathbbm{1}_{\{n\ \text{even}\}}$. As for $(III)$, recall that $N(0,y)=\mathbbm{1}_{\{y=0\}}$, so 
\begin{align*} 
(III) & =\sum_{\substack{1\leq l\leq n \\ n\equiv l \pmod{2}}}\sum_{q=1}^l R(k,q)N(0,l-q) = \sum_{\substack{1\leq l\leq n \\ n\equiv l \pmod{2}}} R(k,l) \\ & = \sum_{\substack{1\leq l\leq n \\ n\equiv l \pmod{2}}} \frac{2^{2-k}}{(k-2)!}\binom{2k-2}{k-1} l^{k-2} + O(n^{k-2}) \\ 
& = \frac{2^{2-k}}{(k-2)!}\binom{2k-2}{k-1} \cdot \frac{1}{2(k-1)} n^{k-1} + O(n^{k-2})
\end{align*}
finishing the proof.
\end{proof}

\begin{cor}
    Given a fixed $k\geq 2$, for $n\gg 1$, a size-$k$ subset $S\subset D_\infty$ maximizing $|S^n|$ has one reflection and $k-1$ translations, and \begin{align*} \gamma^{\max}_{D_\infty}(k,n) & = 1 + \binom{n + k - 2}{k - 2} + \sum_{s=1}^{\min(k-1,n-1)}2^s\binom{k-1}{s}\binom{n-1}{s} \\ & = \frac{2^{k-1}}{(k-1)!} n^{k-1} + O(n^{k-2}).
    \end{align*}
\end{cor}
\begin{proof}
This follows from the previous theorem, as for any $p\geq 2$, we have that
$$\frac{\binom{2p-2}{p-1}}{2^{2p-2}} = \Pr\left(\substack{\text{A random subset of}\ \\ {\{1,\dots,2p-2\}\ \text{has size}\ p-1}}\right) < 1.$$
Hence $\gamma^{\max}_{D_\infty}(k,n)=|S_{k,1,n}^n|$. The result follows from Corollary \ref{cor:1 k n} and from the previous theorem.
\end{proof}

We now shift our attention away from the fixed $k$ case. If we instead assume that $k = n$, then the limit $\lim_{n \to \infty} \gamma_{D_\infty}^{\max}(n, n)$ grows exponentially in $n$. 
Recall the binary entropy $H(\alpha)=-\alpha \log_2 \alpha - (1-\alpha)\log_2(1-\alpha)$. By Stirling's formula, $\log_2 \binom{n}{\floor{\alpha n}} = H(\alpha) n + o(n)$.

\begin{proof}[{Proof of Theorem \ref{thm:k=n}}]
By Corollary \ref{cor:1 k n}
    \[
    \gamma_{D_\infty}(n,1,n)=|S_{n, 1, n}^n| \geq \sum_{s=1}^{n-1} 2^s \binom{n-1}{s}^2 \geq 2^t \binom{n-1}{t}^2
    \]
    for any given $1\leq t\leq n-1$. Let $\alpha = 2 - \sqrt{2}$, and consider $t = \floor{\alpha n}$ (for $n\gg 1$ we have $1\leq t\leq n-1$). Thus
    $$
    \overline{\lim}_{n\rightarrow \infty} \gamma^{\max}_{D_\infty}(n, n)^{1/n} \geq \overline{\lim}_{n\rightarrow \infty} 2^{t/n} \binom{n-1}{t}^{2/n} = 2^{\alpha+2H(\alpha)}.
    $$   
    A straightforward computation now shows that $2^{\alpha+2H(\alpha)}=3+2\sqrt{2}$.

We now turn to bound $\gamma^{\max}_{D_\infty}(n,n)^{1/n}$ from above. 
By Proposition \ref{prop:mult-to-word}, for each $1\leq p\leq k,\ n$:
\[
\Im(\phi_{k,p,n}) \subseteq \bigcup_{l=0}^{n}\left\{(m_1, \dots, m_k) \in \mathbb{Z}^k \;\bigg|\; \sum_{i=1}^{k}|m_i| = l\right\}
\]
so $|S_{k,p,n}^n|=|\Im(\phi_{k,p,n})|\leq \sum_{l=0}^n N(k,l)$. Now for any $T\in (0,1)$ we have
\[
    \sum_{l=0}^{n}N(k,l) = \sum_{l=0}^{n}N(k,l)T^lT^{-l}\leq T^{-n}\sum_{l=0}^{n}N(k,l)T^l \leq T^{-n}G_k(T) = \frac{1}{T^n}\left(\frac{1+T}{1-T}\right)^k.
\]
Hence for $k=n$ we have $$\gamma^{\max}_{D_\infty}(n,n)^{1/n} =\max_{1\leq p\leq n} |S_{p,n,n}^n|^{1/n} \leq \frac{1+T}{T(1-T)}$$ for every $T\in (0,1)$. In particular, for $T=\sqrt{2}-1$ we obtain
$$
\gamma^{\max}_{D_\infty}(n,n)^{1/n} \leq \frac{1+(\sqrt{2}-1)}{(\sqrt{2}-1)(1-(\sqrt{2}-1))} = 3+2\sqrt{2},
$$
as claimed.
\end{proof}

Every virtually cyclic group is either finite-by-$\mathbb{Z}$ or finite-by-$D_\infty$ (see e.g. \cite{St}).

\begin{cor}
Let $G$ be a virtually cyclic group. Then $\gamma^{\max}_G(k,n)\asymp n^{k-1}$ for any fixed $k$, and $\overline{\lim}_{n\rightarrow \infty} \gamma^{\max}_G(n,n)=\begin{cases}
        4 & \text{if $G$ is finite-by-$\mathbb{Z}$}\\
        3+2\sqrt{2} & \text{if $G$ is finite-by-$D_\infty$}\\
    \end{cases}$.
\end{cor}
\begin{proof}
If $1\rightarrow N\rightarrow G\rightarrow H\rightarrow 1$ with $|N|<\infty$ then $$\gamma^{\max}_H(k,n)\leq \gamma^{\max}_G(k,n)\leq |N|\cdot \gamma^{\max}_H(k,n).$$ The result follows.
\end{proof}

\end{document}